\documentclass[11pt]{amsart}
\usepackage{amsfonts,amsmath,amssymb,fullpage,graphicx, comment}
\usepackage{verbatim}
\usepackage{xcolor}
\newtheorem{thm}{Theorem}[section]
\newtheorem{cor}[thm]{Corollary}
\newtheorem{lem}[thm]{Lemma}

\newtheorem{prop}[thm]{Proposition}
\theoremstyle{definition}
\newtheorem{defn}[thm]{Definition}
\theoremstyle{remark}
\newtheorem{rem}[thm]{\bf{Remark}}
\numberwithin{equation}{section}

\newcommand{\beas}{\begin{eqnarray*}}
	\newcommand{\eeas}{\end{eqnarray*}}
\newcommand{\bes} {\begin{equation*}}
	\newcommand{\ees} {\end{equation*}}
\newcommand{\be} {\begin{equation}}
	\newcommand{\ee} {\end{equation}}
\newcommand{\bea} {\begin{eqnarray}}
	\newcommand{\eea} {\end{eqnarray}}
\newcommand{\ra} {\rightarrow}
\newcommand{\txt} {\textmd}

\newcommand{\R}{\mathbb R}

\newcommand{\N}{\mathbb N}

\begin{document}
	
	\title[Quantitative uniqueness properties for quasi-analytic
	functions] {Quantitative uniqueness properties for functions on compact quasi-analytic manifolds}
	
	\author{Mithun Bhowmik and Sanjib Pradhan}
	
	\address{(Mithun Bhowmik) Department of Mathematics, Indian Institute of Technology Kharagpur-721302, India}
	\email{mithun@maths.iitkgp.ac.in}
	
	\address{(Sanjib Pradhan) Department of Mathematics, Indian Institute of Technology Kharagpur-721302, India}
	\email{sanjibpradhan01.24@kgpian.iitkgp.ac.in}

	\thanks{The first author was supported by the PMECRG grant from Anusandhan National Research Foundation(ANRF), Govt. of India, and the second author was supported by the Junior Research Fellowship from IIT Kharagpur, India.}
	
	
	\begin{abstract}
		
		In this article, we establish quantitative uniqueness results for a class of functions defined on a quasi-analytic compact manifold $X$ without boundary. This function class is characterized by the iterates of a positive elliptic linear differential operator on $X$ and, notably, encompasses all functions with a finite spectrum. By employing a relatively dense observable set, we extend classical Logvinenko-Sereda-type results to this quasi-analytic framework. Furthermore, we demonstrate that if these functions satisfy a doubling property, observability holds from any measurable set of positive measure. Our results generalise the propagation of smallness from finite sums of eigenfunctions to infinite sums with an appropriate energy-parameter decay, thereby extending recent findings by Kukavica-Li (Proc. Lond. Math. Soc., 2025) to the quasi-analytic setting.
		
	\end{abstract}
	
	
	
	\maketitle
	
	\section{Introduction}
	
	In this article, we study quantitative uniqueness results for function classes on compact, connected, boundaryless Riemannian manifolds $X$ with quasi-analytic regularity; more generally, on ultradifferentiable compact manifolds. These function classes are defined via the iterates of a positive elliptic linear differential operator $P$. Our main result establishes that the global $L^2$ norm of a function can be controlled by its $L^2$ norm restricted to a relatively dense set, provided its Fourier transform—defined with respect to $P$—exhibits sufficiently rapid decay at infinity. Crucially, these estimates are independent of the specific position of the observable set, depending only on its size and density. Our findings are closely connected to several fundamental areas: spectral inequalities for linear combinations of eigenfunctions typically derived from Carleman estimates \cite{JL}, observability estimates for the heat and Schrödinger equations \cite{AE, AEWZ, HWW} and the classical Logvinenko-Sereda theorem \cite{LS}.
	
	The classical result on uniqueness, or equivalently observability, is given by the work of Paneah \cite{P} and Logvinenko and Sereda  \cite{LS}, which represents a prototypical form of the uncertainty principle. In order to state their
	result, we need a geometric definition.
	\begin{defn}
		Let $\gamma, \ell$ be two positive numbers.  A Borel set $\Omega \subseteq \R^d$ is $(\gamma, \ell)$- relatively dense if $|\Omega \cap Q| \geq \gamma \ell^d$ for any cube $Q \subset \R^d$ of side-length $\ell$. Here, $|A|$ denotes the Lebesgue measure on $\R^d$ for any measurable set $A\subseteq \R^d, d \geq 1$.
	\end{defn} 
	
	\begin{thm}[Paneah-Logvinenko-Sereda Theorem \cite{LS, P}] \label{thm-PLS}
		Fix $\Omega\subseteq \R^d$. For every $N > 0$, there is a constant $C > 0$ such that
		\bes
		\|f\|_{L^2(\R^d)}\leq C \|f\|_{L^2(\Omega)},\:\: \textit{ for every } f \textit{ with } supp (\mathcal F f) \subset B(0, N),
		\ees
		if and only if $\Omega$ is $(\gamma, \ell)$-relatively dense for some $\gamma \in (0, 1)$ and $\ell>0$.  
	\end{thm}
	Here, $\mathcal F f$ denotes the Fourier transform of a function on $\R^d$, and $B(x,r)$ is the Euclidean ball of radius $r$ centred at $x \in \R^d$. Later,  using a different method, Kovrijkine refined the bounds established by Logvinenko-Sereda and achieved the optimal constant-first in the one-dimensional case \cite{K}, and later in his thesis for the multi-dimensional setting. In particular, the constant $C$ can be taken as
	$C= \left(c^d/\gamma \right)^{c d(\ell N+1)}$, for some universal constant $c>0$.

	The uncertainty principle is a meta-theorem in harmonic analysis which states that a function and its Fourier transform cannot both be sharply localized.	
	There are various quantitative mathematical implementations of this principle, for instance, Heisenberg’s uncertainty relation, Hardy’s uncertainty principle, the Paley–Wiener
	theorem, or the framework of annihilating pairs \cite{BDJ, CCR, FS, HJ, T}.  Let $E,B \subset \R^d$. We say that $(E,B)$ is a strong annihilating pair (see, e.g. \cite[Chapter 3]{HJ}) if there exists a constant $C= C(E,B)$
	such that
	\bes
	\|f\|_{L^2(\R^d)}^2 \leq C \left(\|f\|_{L^2(E^c)}^2 +\|\mathcal Ff\|_{L^2(B^c)}^2 \right), \:\: \forall f\in L^2(\R^d). 
	\ees
	Theorem \ref{thm-PLS} implies that the pair $(\Omega, B(0, N))$ is strong annihilating if $\Omega^c$ is relatively dense. 
	
	In \cite{JM}, Jaye-Mitkovski extended the Paneah-Logvinenko-Sereda theorem to functions
	which, instead of being band-limited (i.e. $\mathcal{F} f$ is compactly supported), have sufficiently fast decaying Fourier transforms. In the following, the decay on the Fourier transform is characterized by a weight function $W_\delta(t)=e^{t/\log^\delta(e+ t)}$, for any $\delta\in (0, 1]$. 
	
	\begin{thm} \cite[Theorem 1.3]{JM} \label{thm-JM}
		For every $d \in \N, \gamma \in (0, 1),\delta\in (0, 1], \ell > 0$ and $c > 0$, there exists a finite constant $C = C(d, c, \gamma,\delta, \ell) > 0$ such that if $f \in L^2(\R^d)$ satisfies
		\be \label{est-hatf}
		\int_{\R^d} |\mathcal{F} f(\xi)|^2 ~e^{2|\xi|/log^\delta(e+ |\xi|)} ~d \xi \leq c \|f\|_{L^2(\R^d)}^2,
		\ee
		and $\Omega$ is a $(\gamma, \ell)$-relatively dense set, then
		\bes
		\|f\|_{L^2(\R^d)}\leq C \|f\|_{L^2(\Omega)}.
		\ees
	\end{thm}
	In \cite{JM}, the authors treat a more general class of weight functions. However, for simplicity, we restrict ourselves here to this particular case. The motivation for Theorem \ref{thm-JM} stems from the seminal work of Bourgain-Dyatlov \cite{BD}, in which the authors established the uniqueness theorem on the real line. The idea of the proof of Theorem \ref{thm-JM} is as follows: the estimate (\ref{est-hatf}) on the Fourier transform yields that $f$ belongs to a quasi-analytic class. Using the localization principle behind the
	proof of the Logvinenko-Sereda theorem as presented in \cite{K}, the authors reduced matters to a Remez-type inequality for quasi-analytic functions, which is
	provided by an extension to several variables of a theorem of Nazarov-Sodin-Volberg \cite{NSV}.

	Because of their applications in areas of applied analysis, as it is often used in control theory, the uniqueness results or spectral inequality for the Laplace-Beltrami operator on compact Riemannian manifolds has been studied extensively; see, for instance  \cite{D, JL, LM, OP, LM19, EV, BM, DV, KL} and the references therein.   In \cite{LM19}, Lebeau and Moyano proved spectral estimates for Schr\"odinger operators with analytic potentials, when the metric is
	an analytic perturbation of the identity. More recently, Burq and Moyano \cite{BM} proved spectral estimates for Laplace operators on compact manifolds with Lipschitz metric and density, combining ideas of \cite{LM19} with the propagation of smallness results of Logunov and Malinnikova \cite{LM}. 
	
	Our final departure from the above mentioned results is the result by \cite{KL}, where the authors have studied observability estimates from a measurable set for functions in a Gevrey class satisfying doubling property on a connected and bounded domain in $\R^d$ . As an application, the authors have obtained observability estimates for sum of eigenfunctions of the Laplace-Beltrami operator $\Delta$ on a Gevrey manifold.  
	
	Let $(X, g)$ be a compact and connected Riemannian manifold without boundary that belongs to the $s$-Gevrey class, $s\geq 1$, with the metric $g$, that is also $s$-Gevrey (see Definition \ref{def-X}). Let $ \phi_i$'s  be eigenfunctions of $-\Delta$ with distinct eigenvalues $\lambda_i \geq 1$, that is,   $-\Delta \phi_i=\lambda_i \phi_i$. Let $h= \sum_{i=1}^m \phi_i$.  Then 
	\begin{thm} \cite[Theorem 2.2]{KL} \label{thm-intro}
		There exists a constant $C\geq 1$ such that for every measurable subset $E\subset X$ with positive measure, we have 	
		\bes
		\|h\|_{L^{\infty}(X)}\leq  \left(\frac{C}{|E|}\right)^{C\gamma} \gamma^{C(s-1)\gamma} \|f\|_{L^{\infty}(E)},
		\ees
		where $\lambda= \max \{\lambda_1, \cdots, \lambda_m\}$ and $\gamma= C\sqrt{\lambda} +C m^2 \log m +C$ is the doubling index of $h$.
	\end{thm}
	
	To the best of our knowledge, there are no results on quantitative uniqueness for functions beyond the finite spectrum on Riemannian manifolds. In this paper, we address this question. In light of the above discussion, the natural setting for such a problem should be a ultradifferentiable manifold (in particular, Gevrey and quasi-analytic manifolds).

	Let $X$ be a compact connected smooth manifold without boundary of dimension $d$ with a fixed volume element $dx$. Let $P$ be a positive elliptic linear differential operator of order $m$ with smooth coefficients on $X$. We now review some basic results on Fourier analysis relative to $P$. These are standard and can be found in papers of Seeley \cite{See1, See2}. 
	
	The eigenvalues of $P$ form a sequence $\{\lambda_j\}$, where each eigenvalue $\lambda_j$ is associated with a finite dimensional eigenspace $H_j$ consisting of $C^\infty$ functions on $X$. We put $d_j=\dim H_j$ and take $H_0=ker(P)$ and $\lambda_0=0$. We may assume that $\lambda_j $'s are ordered as  
	$0 < \lambda_1 < \lambda_2 < \cdots.$ Let $\{\phi_j^k\},~1 \leq k \leq d_j$ denote the eigenfunctions of $P$ with respect to the eigenvalues $\{\lambda_j\}$. From now on, we enumerate the collection $\{\phi_j^k\},~1 \leq k \leq d_j$ as a single sequence $\{\phi_j\}$. Then $ \{\phi_j\} $ form an orthonormal basis in $L^2(X)$. For $ f \in L^2(X) $, let us denote $f_j = \langle f,\phi_j \rangle_{L^2(X)}$, then by Plancherel formula, we get
	\be \label{Plancherel}
	\|P^{k} f\|_{L^2(X)}=  \left(\sum_{j \in \mathbb N_0} \lambda_j^{2k} |f_j|^2  \right)^{1/2} . 
	\ee
	Here, we note that the ellipticity of $P$ and the Plancherel formula imply the following characterization of
	smooth functions in terms of their Fourier coefficients:
	\be \label{est-smooth}
	f=\sum_{j \in \N_0} f_j \phi_j~  \textit{is smooth} \iff \forall N\in \N,~  \exists C_N>0  \textit{ such that }   |f_j|\leq C_N \lambda_j^{-N}, \text{ for all } j.
	\ee
	If $P$ is an analytic differential operator on an analytic manifold $X$, the result of Seeley \cite{See2} can be reformulated as
	\bes
	f=\sum_{j \in \N_0} f_j \phi_j \textit{ is analytic} \iff \exists s>1,~ c>0  \textit{ such that }   |f_j|\leq c s^{-(\lambda_j)^{1/m}}, \text{ for all } j.
	\ees
	
	In this paper, we consider  a characterization similar to the above for
	classes of functions in between smooth and analytic functions, namely, for quasi-analytic and Gevrey
	functions. This was first considered by Komatsu \cite{Kom}, and has been studied further by many others, for example, \cite{AlJoOl,Fu,KOM,KrMiRa,DR,DR16}. We start by reviewing the Gevrey and quasi-analytic functions on a domain in $\R^d$. Throughout this article, $U$ will denote an open subset of $\mathbb{R}^d$.
	
	\begin{defn}[Gevrey class] 
		The Gevrey class $G^s(U)$ is defined as the space of smooth functions $f$ on $U$ such that for each compact $K\subset U$ there exists $L\geq 1, h>0$  such that for all  multi-indices $\alpha \in \N_0^d$, one has
		\bes
		\|\partial^\alpha f\|_{L^\infty(K)} \leq L h^{|\alpha|} |\alpha|!^s,
		\ees
		where $s \geq 1$ and $|\alpha|= \alpha_1+ \cdots+\alpha_d$. We refer to such functions as a $s$-Gevrey functions. When $s=1$, we have $G^{1}(U)$, which is the class of analytic functions on $U$. 
	\end{defn}

	In fact, there are more general classes of smooth functions associated with an increasing sequence $\mathcal M=\{M_k\}_{k\in \N_0}$ of positive numbers. 
	
	\begin{defn}[Ultradifferentiable class] \label{Defn-C} 
		The ultradifferentiable class $\widetilde{C}_{\mathcal M}(U)$ is the space of smooth functions $f$ on $U$ such that for each compact $K\subset U$ there exists $L\geq 1, h>0$ such that for all $\alpha \in \N_0^d$, 
		\be \label{def-cm}
		\|\partial^\alpha f\|_{L^\infty(K)} \leq L h^{|\alpha|} M_{|\alpha|}.
		\ee
	\end{defn}
	
	The sequence $\mathcal M=\{M_k\}_{k\in \N_0}$ is called the weight sequence. Clearly, $s$-Gevry class corresponds to the sequence $M_k=(k!)^s$ for $s \geq 1$.  We set $\mathfrak m=\{m_k\}_{k\in \N_0}$, where $m_k=M_k/(k!)$ for all $k\in \N_0$.
	Following \cite{Fu}, we say $\mathcal M=\{M_k\}_{k \in \N_0}$ is a regular weight sequence if
	\begin{enumerate}
		\item[(i)] $\quad  m_0=1$.
		\item[(ii)] $\quad  m_k$ is increasing.
		\item[(iii)] $\quad m_k^2\leq m_{k-1} m_{k+1}, \:\: \forall k\in \N$ \:\:  (logarithmic convexity). 
		\item[(iv)] $\quad  \sup_{k \in \mathbb{N}^{>0}} \left( \frac{m_{k+1}}{m_k} \right)^{1/k} < \infty  \:\: \text{(stability)}. $  
	\end{enumerate}
	In addition, a sequence $\mathfrak m$ has moderate growth if 
	\be \label{eqn-mg}
	\sup_{k,p \in \mathbb{N}^{>0}} \left( \frac{m_{k+p}}{m_k m_p} \right)^{1/(k+p)} < \infty.	
	\ee
	Conditions (i)-(iii) imply that the class $\widetilde{C}_{\mathcal M}(U)$ is closed under addition, multiplication and composition.	Condition (iv) guarantees that the class is closed under derivation \cite{KrMiRa}. The moderate growth condition (\ref{eqn-mg}), in addition to regularity, would provide $\widetilde{C}_\mathcal{M}$-hypoellipticity for elliptic partial differential operators with coefficients in $\widetilde{C}_\mathcal{M}$. 
	
	Since we want our space should include real-analytic functions, we assume that
	\be \label{M-con}
	\textit{There exists } \mathfrak l>0, \textit{ such that } k!\leq {\mathfrak l}^k M_k,  \; \; \; \textit{for all $k\in \N_0$}.
	\ee
	\begin{defn}[Quasi-analytic class]
		A class $\mathcal S$ of smooth functions on $U$ is said to be quasi-analytic class if whenever a function $f\in \mathcal S$, and all its partial derivatives vanish at a point in $U$, then $f \equiv 0$ on $U$.
	\end{defn}
	By a classical result of Denjoy and  Carleman \cite[Theorem 19.11]{Ru}, the class $\widetilde{C}_{\mathcal M}(0, 1)$ with a log-convex sequence $\mathcal M=\{M_k\}_{k \in \N_0}$ is quasi-analytic,  if and only if 
	\be \label{quasi-cond}
	\sum_{k=1}^\infty \frac{M_{k-1}}{M_k}=\infty.
	\ee 
	Whenever a sequence $\{M_k\}_{k \in \N_0}$ satisfies \eqref{quasi-cond}, we say that the sequence is quasi-analytic. 
	For example, the sequence $M_k= k! \log^{sk}( e+k)$ produces a quasi-analytic class if $0< s \leq 1$.
	As a corollary of Theorem \ref{thm-main}, we show that if the regular weight sequence $\{M_k\}_{k \in \N_0}$ satisfies the moderate growth condition \eqref{eqn-mg} and condition \eqref{M-con}, then the ultradifferentiable class $\widetilde{C}_{\mathcal M}(U)$ is quasi-analytic whenever \eqref{quasi-cond} holds.
	
	We call a map $f=(f_1, \cdots, f_d): U \subset \R^d \ra \R^d$ is in $\widetilde{C}_{\mathcal M}(U: \R^d)$ if each $f_j\in \widetilde{C}_{\mathcal M}(U)$. For simplicity of notation, we will also denote this class by $\widetilde{C}_{\mathcal M}(U)$.
	
	\begin{defn} \label{def-X}
		Let $(X, g)$ be a smooth Riemannian manifold and $\mathcal M = \{M_k\}_{k \in \N_0}$ a regular weight sequence. We say that $X$ is an ultradifferentiable manifold of class $\widetilde{C}_{\mathcal M}$ if the following holds:
		\begin{enumerate}
			\item[(i) ] There exists an atlas $\mathcal{A}$ of $X$ consisting of charts such that
			\bes
			\varphi' \circ \varphi^{-1} \in \widetilde{C}_{\mathcal M}, \:\: \textit{ for all } \:\: \varphi, \varphi' \in {\mathcal A}. 
			\ees
			\item[(ii)] The metric $g$ also belongs to this class, which is equivalent to $g_{ij}\in \widetilde{C}_{\mathcal M}(U)$ on every ultradifferentiable chart $(U, \phi)$.
		\end{enumerate}
	\end{defn}
	Example of such compact manifold can be given using  the implicit function theorem \cite{KOM}. 
	If we take $M_k = (k!)^s$, then $X$ is called an $s$-Gevrey manifold. On the other hand, if we take a regular sequence satisfying the quasi-analytic condition \eqref{quasi-cond}, then $X$ is called quasi-analytic manifold. 
	
	So far, we have considered only the local version of regular classes. We now turn to the global setting, following the approach inspired by Komatsu’s definition \cite{Kom};	see also \cite{DR16}.
	\begin{defn}	
		We introduce the class of smooth functions on $X$ associated with a sequence $\mathcal M=\{M_k\}_{k \in \N_0}$ 
		\be \label{defn-CM}
		C_{\mathcal M}(X)=\left\{f\in C^\infty(X): \| P^k f\|_{L^2(X)} \leq M_{mk} \|f\|_{L^2(X)},  \:\: \forall k \in \N_0 \right\}.
		\ee
	\end{defn}	
	\begin{rem}	
		The advantage of above definition, as mentioned in \cite{DR16}, is that we do not refer to local coordinates to
		introduce the class $C_{\mathcal M}(X)$. This allows for a definition of analogues of analytic, or Gevrey functions even if the manifold $X$ is ‘only’ smooth. For example, by taking $M_k =c^k k!$, we obtain the class $C_{\mathcal M}(X)$ of functions satisfying the condition
		\bes
		\|P^k f\|_{L^2(X)} \leq c^{mk} (m k)!,\; \;\; k=0,1,2, \cdots .
		\ees	
		If $X$ and $P$ are analytic, this is precisely the class of analytic functions on $X$. 
	\end{rem}		
	In this article, we establish a uniqueness result for functions in $C_{\mathcal M}(X)$ for quasi-analytic manifold $X$. Let us first introduce the notion of a relatively dense subset in the setting of compact manifolds. For $x\in X$, let $\mathcal B(x, r)$ denote a geodesic ball of radius $r$ centred at $x$.  We write $r_{\mathrm{inj}} > 0$ for the injectivity radius of $X$.
	
	\begin{defn} \label{defn-dense} Let $\gamma \in (0, 1]$.  A subset $\Omega \subset X$ is said to be $\gamma$-relatively dense if there exists $r \in (0, r_{inj}/2)$ such that  
		\bes
		| \mathcal B(x, r) \cap \Omega | \geq \gamma ~ |\mathcal B(x,  r)|,  \:\: \text{  for all } \:\: x\in X. 
		\ees 
		Here, for $A \subseteq X$, $|A|$ denote the measure of $A$ with respect to the volume element on $X$.
	\end{defn}
	
	The following is our main uniqueness result on a compact connected quasi-analytic manifold $X$ associated to a quasi-analytic regular weight sequence $\mathcal M = \{M_k\}_{k \in \N_0}$ satisfying \eqref{eqn-mg} and \eqref{M-con}. 
	\begin{thm} \label{thm-main}
		Let $\gamma\in(0, 1)$. There exists an absolute constant $C=C(X)$ such that 
		\bes
		\|f\|_{L^2(X)}\leq  \frac{C}{\sqrt{\gamma}} \left(\frac{2d}{\gamma} \Gamma(2n_{\widetilde{\mathcal M}, s})\right)^{2 n_{\widetilde{\mathcal M}, s}} \|f \|_{L^2(\Omega)},  \:\: \textit{ for all } f\in   C_{\mathcal M}(X),
		\ees
		for all  $\gamma$-relatively dense subset $\Omega$ in $X$.  Here,  $\widetilde{\mathcal M}$ and $s$ are defined in (\ref{defn-tilde-M}) and (\ref{defn-s}), respectively.  Also,  the functions $n_{\mathcal M,  s},$ and  $\Gamma$  are defined in (\ref{defn-Gamma}) and (\ref{defn-ns}), respectively.  
	\end{thm}
	As a byproduct, we obtain the following analogue of Chernoff's theorem \cite[Theorem 6.1]{Ch} in the setting of compact connected quasi-analytic manifolds.
	
	\begin{cor} \label{unique-cont}
		Any nonzero function $f\in C_{\mathcal M}(X)$ cannot vanish on a positive measure set.
	\end{cor}
	Analogues of this corollary have been studied extensively in recent years; see, for instance, \cite{BP, BPR0, BPR, GT21, Sarkar}.
	
	As an application of Theorem~\ref{thm-main}, we establish the following analogue of Theorem~\ref{thm-JM} in the setting of quasi-analytic manifolds \(X\) associated with the regular weight sequence $M_k =k!~ (\log(k+e))^{\delta k}$,
	where $\delta \in (0,1]$.

	\begin{thm}\label{QM}
		Let $\gamma \in (0,1]$ and $c$ be positive constants. Then, there exists $C=C(\gamma, c, X) > 0$ such that the following holds:
		whenever a function $f\in L^2(X)$ satisfies 
		\bes
		\sum_{j\in \N_0} |f_j|^2 ~e^{2j^{1/m}/ \log^\delta(e+j^{1/m})}  \leq c^2 \|f\|_{L^2(X)}^2,
		\ees
		and $\Omega$ be any $\gamma$-relatively dense set, we obtain
		\bes
		\|f\|_{L^{\infty}(X)} \leq C \|f\|_{L^\infty(\Omega)}.
		\ees
	\end{thm}
	
	We now consider the functions in $C_{\mathcal M}(X)$ which satisfies the doubling condition. 	Let us recall the definition first. A smooth function $f$ is said to  satisfy the doubling condition if there exist $\kappa \geq 2, r_0>0$ such that 
	\be \label{doubling}
	\sup_{\mathcal B(x,2r)} |f|\leq \kappa \sup_{\mathcal B(x,r)} |f|, \:\: \textit{ for all } \:\:  r\in (0, r_0], \; \;  x\in X. 
	\ee
	The doubling condition is known to hold for the sum of eigenfunctions of the Laplace operator on analytic manifolds \cite{D}.
	
	Our second result generalizes the results of \cite{BM} and \cite{KL} to the setting of compact connected ultradifferentiable manifolds $X$ associated with a regular weight sequence $\mathcal M$ satisfying \eqref{eqn-mg} and \eqref{M-con}.
	\begin{thm} \label{thm-main-1}
		Let $f\in C_{\mathcal M}(X)$ satisfy the doubling condition (\ref{doubling}) with the parameter $\kappa \geq 2$. Then there exists positive constants $C=C(\mathcal M,X),~ \mathcal C=\mathcal C(X)$ such that for any positive measure set $\omega\subset X$
		\bes
		\|f\|_{L^{\infty}(X)} \leq \kappa^{3\mathcal C+1} \left(\frac{C}{|\omega|}\right)^{2n_0} M_{n_0}  \|f\|_{L^\infty(\omega)},
		\ees
		where $n_0= 2[\max \{\log_2 \kappa, \mathfrak l e (hr_0)^{-1}\}] +1$, $h>0$ depends only on $\mathcal M$ and $X$.
	\end{thm}
	Let $X$ be an analytic, or Gevrey manifold, and $ \phi_i$'s  be eigenfunctions of $P$ with distinct eigenvalues $\lambda_i \geq 1$, that is,   $P \phi_i=\lambda_i \phi_i$.
	
	\begin{cor}	\label{cor-eig}
		Let $h=\sum_{i=1}^N \phi_i$. Then there exists $C \geq 1$ (depends on $\mathcal M,~ X$) such that for any measurable subset $\omega \subset X$ with positive measure, we get
		\bes
		\|f\|_{L^{\infty}(X)} \leq \left(\frac{C}{|\omega|}\right)^{2\gamma} \gamma^{C\gamma} \|f\|_{L^{\infty}(\omega)},
		\ees
		where $\gamma= C \sqrt{\lambda_N} + C N^2~ \log N +C$ is the doubling index of $h$.
	\end{cor}
	
	\section{Some Auxiliary Results}
	Throughout this section $(X, g)$ be a compact connected quasi-analytic manifold associated to a regular weight sequence $\mathcal M = \{M_k\}_{k \in \N_0}$ (i.e, $m_k = M_k/(k!)$ satisfies the hypothesis ${(i)-(iv)}$) satisfying \eqref{eqn-mg} and \eqref{M-con}, equipped with a maximal atlas ${(U_k,\varphi_k)}$ consisting of quasi-analytic charts.  
	
	Let us consider an elliptic linear partial differential operator
	\bes
	P(x, \partial)= \sum_{|\beta|\leq m} a_\beta(x)  \partial^\beta
	\ees 
	of order $m$ with the coefficients $\left.\alpha_{\beta}\right|_{U_k} \in \widetilde{C}_{\mathcal M}(U_k)$ for all $\beta \in \N_0^d$. 	
	
	The following result provides estimates for all partial derivatives of $f\in C_{\mathcal M}(X)$. 
	\begin{lem}\cite[Theorem 2.3.]{DR16} \label{lem-D} \label{lem-DR}
		Let $f\in C_{\mathcal M}(X)$. Then for every quasi-analytic chart $(U, \phi)$, $f\in \widetilde{C}_{\mathcal M}(U)$. Precisely, for every compact set $K\subset U$, there exists $L\geq 1, h>0$ (depending on $K$ and $\mathcal M$) such that for $f\in C_{\mathcal M}(X)$, the following holds
		\bes
		\|\partial^{\alpha} f\|_{L^\infty(K)}\leq L h^ {|\alpha|} M_{|\alpha|}, \:\: \textit{ for all } \alpha \in \mathbb{N}_0^d.
		\ees
	\end{lem}
	Here, $\partial^{\alpha}$ denotes the mixed partial derivative of order $|\alpha|$ in local-coordinate charts.
	Throughout this section,  $\mathcal N  = \{N_k\}_{k \in \N_0}$ be any log-convex sequence satisfying the quasi-analyticity condition (\ref{quasi-cond}) and $N_0=1$.  For a ball $B\subset \R^d$, $d\geq 1$, we refine the definition of the quasi-analytic class on $B$.  Precisely,
	\bes
	C_{\mathcal N} (B)=\{h\in C^\infty(B): \|h^{(k)}\|_{L^\infty(B)}\leq N_k, \:\: \forall k\in \N_0\}.
	\ees 
	Here,  $h^{(n)}$ denotes the $n$-th derivative of $h$.
	For $h \in C_{\mathcal N} (B)$, we define $n_h$  by
	\bes
	n_h= \max \left\{n\in \N : \sum_{\log \|h\|_{L^\infty(B)}^{-1}<k \leq n} \frac{N_{k-1}}{N_k} < e \right\}.
	\ees
	For $d=1$,  and $h\in C_{\mathcal N}(0, 1)$, the quantity $n_h$	 is the well-known Bang degree of $h$, which controls the number of zeros (with counting multiplicities) of a function in $C_{\mathcal N}(0, 1)$, therefore,  plays the same role as the degree of polynomials \cite{NSV}.  Clearly, $n_h$ depends on both the decay of the ratios of $N_{k-1}/N_k$ and a lower bound for $\|h\|_{L^\infty(B)}$. Since $\mathcal N$ satisfies the quasi-analytic condition (\ref{quasi-cond}), the quantity $n_h$ is always finite.  For our purposes, we will want uniform bounds on $n_h$ for arbitrary functions $h \in C_\mathcal N (B)$. To achieve this,  we set, for $s \in (0, 1]$
	\be \label{defn-Gamma}
	n_{\mathcal N, s}= \max \left\{n\in \N: \sum_{-\log s <k \leq n} \frac{N_{k-1}}{N_k} < e \right\}.
	\ee
	Therefore, if $h \in C_\mathcal N(B)$ satisfies $\sup_{t \in [0, 1]} |h(t)| \geq s$, then $n_h\leq n_{\mathcal N, s}$.  As in  \cite{JM, NSV}, we also define
	\be \label{defn-ns}
	\gamma_{\mathcal N}(k)= \sup_{1\leq \ell\leq k} \ell\left(\frac{N_{\ell+1}N_{\ell-1}}{N_\ell^2}-1 \right), \:\: \textit{ and }~  \; \Gamma(k)=4e^{4+4\gamma_{\mathcal N}(k)}.
	\ee
	
	The following result establishes a uniform bound on the $L^2$ norm of a quasi-analytic function over a ball $B\subset \R^d$ in terms of the $L^2$ norm of its restriction to a subset of positive measure. This is a higher dimensional extension of \cite[Theorem B]{NSV}. 	
	\begin{prop}\cite[Prop. A10]{M} \label{prop-martin}
		Let $0<\gamma \leq 1$ and $0<t \leq 1$. If $E\subset B$  is a measurable subset satisfying $|E| \geq \gamma |B|>0$, then for all  $f\in C_{\mathcal N} (B)$ with $\|f\|_{L^\infty(B)} \geq t$,
		\bes
		\int_B |f(x)|^2~dx \leq \frac{2}{\gamma} \left(\frac{2d}{\gamma} \Gamma(2n_{\mathcal N, t}) \right)^{4n_{\mathcal N, t}} \int_{E} |f(x)|^2~dx.
		\ees
	\end{prop}
	
	In the proof of the above theorem, the authors make essential use of the fact that 
	$\{N_k\}_{k \in \N_0}$  is a quasi-analytic sequence. It is well known that Gevrey classes 
	$G^s$
	are not quasi-analytic when 
	$s>1$. Nevertheless, in \cite{KL}, Kukavica and Li established a similar result for the class $G^s(\Omega)$, where $\Omega$ is a $C^1$ domain in $\R^d$, under the additional assumption that the functions in this class satisfy the doubling condition. Their proof generalizes to $C_{\mathcal M}(\Omega)$, for any sequence $\mathcal M=\{M_k\}_{k \in \N_0}$ satisfying some certain growth condition. For the reader’s convenience, we also provide a streamlined proof following their argument.
	\begin{lem} \label{lem-KL}
		Let $\Omega$ be a connected bounded $C^1$ domain in $\R^d$ and let $\mathcal M=\{M_k\}_{k \in \N_0}$ be a sequence of positive numbers satisfies \eqref{M-con}. 
		Let $f\in \widetilde{C}_{\mathcal M} (\Omega)$ satisfy the doubling condition: there exists $\kappa\geq 2, r_0>0$ such that 
		\be \label{doubling-lem}
		\|f\|_{L^\infty(B(x,2r) \cap \Omega )}\leq \kappa \|f\|_{L^\infty(B(x,r) \cap \Omega)}, \:\: \forall r\in (0, r_0], \:\: x\in \Omega.
		\ee
		Then, for any measurable set $\omega\subset \Omega$ with positive measure, we have
		\bes
		\|f\|_{L^\infty(\Omega)} \leq \hat{C} \|f\|_{L^\infty(\omega)},
		\ees  
		where 
		\bes
		\hat{C}=  \kappa^{2 \mathfrak C} {C}^{n_0} \max\{1, h\}^{n_0+1} LM_{n_0+1} \left(\frac{|\Omega|}{|\omega|}\right)^{2n_0}   \|f\|_{L^\infty(\omega)},
		\ees
		where $n_0= 2[\max \{\log_2 \kappa, \mathfrak l e (hr_0)^{-1}\}] +1,\; C, \; \mathfrak C,~ L$ and $h$ are positive constants depending on $\mathcal M, ~\Omega$.
	\end{lem}
	
	\begin{proof}
		Since $f\in \widetilde{C}_{\mathcal M} (\Omega)$ and $\Omega$ is bounded, it follows from the Definition \ref{Defn-C} that there exists $L\geq 1, h>0$ such that for all multi-indices  $\alpha \in \N^d_0$,
		\bes
		\|\partial^{\alpha}f\|_{L^{\infty}(\Omega)} \leq L h^{|\alpha|}M_{|\alpha|} \|f\|_{L^{\infty}(\Omega)}.
		\ees
		The term $\|f\|_{L^{\infty}(\Omega)}$ is included on the right-hand side to make similarity with the notation in \cite{KL}. Following the proof of \cite[Theorem 2.1]{KL}, we get from the inequality (3.14) in \cite{KL} that 
		\be \label{est-KL}
		\|f\|_{L^\infty(\Omega)} \leq \frac{\kappa^{\mathfrak C}}{r^{\log_2 \kappa}} C^{n+1} \left(\left(\frac{|\Omega|}{|\omega|}\right)^n \|f\|_{L^\infty(\omega)} + L (rh)^{n+1}  M_{n+1} \|f\|_{L^\infty(\Omega)}\right),
		\ee
		for all $n\in \N$ and $r\in (0, r_0]$. We now choose 
		\bes
		r=\left(\frac{ \|f\|_{L^\infty(\omega)}}{L h^{n+1} M_{n+1} \|f\|_{L^\infty(\Omega)}}\right)^{1/(n+1)}.
		\ees
		We later choose $n$ large enough so that $r< r_0$. Using the above form of $r$, it follows from the estimate (\ref{est-KL})  that
		\bea \label{est-KL-2}
		\|f\|_{L^\infty(\Omega)} &\leq& \kappa^{\mathfrak C} C^{n+1} \left(\left(\frac{|\Omega|}{|\omega|}\right)^n +1\right) \|f\|_{L^\infty(\omega)}  \left(\frac{ \|f\|_{L^\infty(\omega)}}{ L h^{n+1}M_{n+1} \|f\|_{L^\infty(\Omega)}}\right)^{-\log_2 \kappa/(n+1)} \nonumber\\
		&\leq& \kappa^{\mathfrak C} C^{n+1} \left(\frac{|\Omega|}{|\omega|}\right)^n \|f\|_{L^\infty(\omega)}^{1- \log_2 \kappa/(n+1)} \|f\|_{L^\infty(\Omega)}^{\log_2 \kappa/(n+1)} \left(\frac{1 }{L h^{n+1}M_{n+1} }\right)^{-\log_2 \kappa/(n+1)}.
		\eea
		Using the fact that $n!\leq {\mathfrak l}^n M_n$ ,we deduce that $r \leq \frac{\mathfrak l e }{h(n+1)}$, so therefore by choosing 
		
		\bes
		n_0= 2[\max \{\log_2 \kappa, \mathfrak l e (hr_0)^{-1}\}] +1,
		\ees
		where $[m]$ denotes the largest integral part of $m$, we obtain $r\leq r_0$. Moreover,  $n_0\geq 2 [\log_2 \kappa]+1 $ insures that $\eta=\log_2 \kappa/(n_0+1) \in (0, 1/2)$. Therefore, estimate (\ref{est-KL-2}) yields 
		\bes 
		\|f\|_{L^\infty(\Omega)} \leq \kappa^{\mathfrak C} C^{n_0+1} \left(\frac{|\Omega|}{|\omega|}\right)^{n_0} \|f\|_{L^\infty(\omega)}^{1-\eta} \|f\|_{L^\infty(\Omega)}^{\eta} \left(\frac{1}{L h^{n_0+1}M_{n_0+1}}\right)^{- \eta}.
		\ees
		The fact that $1<1/(1-\eta)< 2$ implies  that
		\beas
		\|f\|_{L^\infty(\Omega)} &\leq& \left(\kappa^{\mathfrak C} C^{n_0+1} (L h^{n_0+1}M_{n_0+1})^{\eta} \right)^{1/(1-\eta)}  \left(\frac{|\Omega|}{|\omega|}\right)^{n_0/(1-\eta)}  \|f\|_{L^\infty(\omega)} \\
		&\leq& \kappa^{2 \mathfrak C} {C}^{n_0} \max\{1, h\}^{n_0+1} LM_{n_0+1} \left(\frac{|\Omega|}{|\omega|}\right)^{2n_0}   \|f\|_{L^\infty(\omega)}.
		\eeas
	\end{proof}

	\section{Proof of the main results}

	In this section we prove the results. 
	
	\begin{proof}[Proof of Theorem \ref{thm-main}]
		Let $X$ be a compact connected quasi-analytic manifold of dimension $d$. Let $f\in C^\infty(X)$ such that $\|P^k f\|_{L^2(X)}\leq  M_{mk} \|f\|_{L^2(X)}$. Without loss of generality,  we assume $\|f\|_{L^2(X)}=1$.  Since $\Omega$ is a $\gamma$-relatively dense set,  there exists $r \in (0,  {r_{inj}}/{2})$ such that $|\mathcal B(x, r) \cap \Omega| \geq \gamma | \mathcal B(x, r)|$, for all $x\in X$.  We find a finite collection $x_j$ in $X$ such that $X=\cup_{j=1}^N \mathcal B(x_j,  r)$. From now onward, we use $\mathcal B_j$ to denote $\mathcal B(x_j, r)$.  We note that 
		\be \label{est-GB}
		N^{-1} \sum_{j=1}^N \|f\|_{L^2(\mathcal B_j)} \leq  \|f\|_{L^2(X)} \leq \sum_{j=1}^N  \|f\|_{L^2(\mathcal B_j)}.
		\ee
		Since $f \in C_{\mathcal M}(X)$, it follows from Lemma \ref{lem-DR} that 
		\be \label{D-est}
		\|\partial^\alpha f\|_{L^2(X)} \leq \|\partial^\alpha f\|_{L^\infty(X)}\leq L h^{|\alpha|} M_{|\alpha|},  \:\: \textit{ for all } \alpha \in \mathbb{N}_0^d.
		\ee
		Here, the constants $L\geq 1, h>0$ depends only on $X$ and $\mathcal M$. We decompose the family $\{\mathcal B_j\}_{j=1}^{N}$ into good and bad balls. A ball $\mathcal B_j$ is said to  be bad if there exists $\alpha\in \N_0^d$ with $|\alpha| \geq 1$ such that 
		\bes
		\int_{\mathcal B_j} \left| \partial^\alpha f(x)\right|^2 ~dx > 2^{2|\alpha|+d+1}N L^2 {h}^{2|\alpha|} M_{|\alpha|}^2 \int_{\mathcal B_j} |f(x)|^2
		~dx.
		\ees
		If a ball is not bad,  we call it good.  Precisely,  if $\mathcal B_j$ is good,  then for all $\alpha\in \N_0^d$ with $|\alpha| \geq 1$
		\be \label{defn-good}
		\int_{\mathcal B_j} \left| \partial^\alpha f(x)\right|^2 ~dx \leq 2^{2|\alpha|+d+1}N L^2 {h}^{2|\alpha|} M_{|\alpha|}^2  \int_{\mathcal B_j} |f(x)|^2~dx.
		\ee
		It follows from the definition that if $\mathcal B_j$ is bad,  then 
		\be \label{est-bad}
		\int_{\mathcal B_j}  |f(x)|^2~dx < \sum_{\alpha\in \N_0^d}\frac{1}{2^{2|\alpha|+d+1}N L^2 {h}^{2|\alpha|} M_{|\alpha|}^2}  \int_{\mathcal B_j}  |\partial^\alpha f(x)|^2~dx.
		\ee
		Let $\mathcal B_{bad}=\cup \mathcal B_j$,  where the balls $\mathcal B_j$'s are bad.  Similarly,  $\mathcal B_{good}= \cup \mathcal B_j$, where $ \mathcal B_j$'s are good balls.
		Consequently, the estimates (\ref{est-bad}), (\ref{est-GB}) and \eqref{D-est} yield
		\beas
		\int_{\mathcal B_{bad}}  |f(x)|^2~dx &\leq&  \sum_{\{j: \mathcal B_j \textit{ is bad }\}} \int_{\mathcal B_j} | f(x)|^2~dx\\
		&\leq& N \sum_{\alpha\in \N_0^d}\frac{1}{2^{2|\alpha|+d+1}N L^2 {h}^{2|\alpha|} M_{|\alpha|}^2} \int_X | \partial^\alpha f(x)|^2~dx\\
		&\leq&  \sum_{m=0}^{\infty} \binom{m+d-1}{m}~\frac{N}{2^{2m+d+1}N L^2 {h}^{2m} M_{m}^2}~ L^2 h^{2m} M_{m}^2 \\
		&\leq& \sum_{m=0}^{\infty} \frac{2^{m+d-1}}{2^{2m+d+1}} = \frac{1}{2}\|f\|_{L^2(X)}^2,
		\eeas
		since $\binom{m+d-1}{m} \leq \sum_{j=0}^{m+d-1} \binom{m+d-1}{j} = 2^{m+d-1}$.
		Consequently,
		\be \label{est-good}
		\int_{\mathcal B_{good}} |f(x)|^2~ dx  \geq \frac{1}{2} \|f\|_{L^2(X)}^2.
		\ee
		We now fix a good ball $\mathcal B_j$,  and henceforth we will denote it by $\mathcal B$,  omitting the subscript $j$.  Let $x_0 \in \mathcal B$ be a point such that $|f(x_0)| \geq \|f\|_{L^2(\mathcal B)} ~|\mathcal B|^{-1/2}$.  This is possible and can be shown by contradiction.
		By Sobolev embedding theorem  \cite[Theorem 4.12]{AF},
		\bes
		W^{d,2}(\mathcal B) \hookrightarrow L^\infty(\mathcal B),
		\ees
		implies that there exists a positive constant $C_d \ge 1$ depending only on the dimension \(d \ge 1\) such that
		
		\bes
		\|f\|_{L^\infty(\mathcal B)} \leq C_d \sum_{|\beta|\leq d}\|\partial^\beta f\|_{L^2(\mathcal B)},	
		\ees
		and hence, since $\mathcal B$ is a good ball and $\mathcal M$ satisfies the stability condition ${(iv)}$, we obtain
		\beas
		\|\partial^\alpha f\|_{L^\infty(\mathcal B)} &\leq& C_d \sum_{|\beta|\leq d}\|\partial^{\alpha+\beta} f\|_{L^2(\mathcal B)}\\
		&\leq& C_d \sum_{|\beta|\leq d} 2^{(|\alpha|+|\beta|)+\frac{d+1}{2}}N^{\frac{1}{2}} L {h}^{|\alpha|+|\beta|} M_{(|\alpha|+|\beta|)} \|f\|_{L^2(\mathcal B)}\\
		&\leq& C_d \sum_{|\beta|\leq d} 2^{(|\alpha|+|\beta|)+\frac{d+1}{2}}N^{\frac{1}{2}} L {h}^{|\alpha|+|\beta|} H^{\frac{2d|\alpha| +d^2 - d}{2}}M_{|\alpha|} \|f\|_{L^2(\mathcal B)}\\
		&\leq& C_d \left(\sum_{|\beta|\leq d} 1 \right) L N^{\frac{1}{2}}2^{\frac{3d+1}{2}} h^d H^{\frac{d^2-d}{2}} (2hH^{2d})^{|\alpha|} M_{|\alpha|}\|f\|_{L^2(\mathcal B)}\\
		&=& C \mathcal{H}^d \mathfrak{B}^{|\alpha|}  M_{|\alpha|} \|f\|_{L^2(\mathcal B)}.
		\eeas 	
		where $C = C_d (\sum_{|\beta|\leq d} 1)L N^{1/2}$, $\mathcal{H} = 2^{{(3d+1)}/{2d}}h H^{(d-1)/2}$ and $\mathfrak{B} = 2hH^{2d}$ all these constants depends only on $X$ and the weight sequence $\mathcal M$; moreover, the constant $H$ arises from the stability condition $(iv)$ satisfied by $\mathcal M$.	
		Let 
		\bes
		F(x)=\frac{f(x)}{C \mathcal{H}^d \|f\|_{L^2(\mathcal B)}}, \:\: \textit{ for } \:\: x\in \mathcal B. 
		\ees						
		From above we get that
		\bes
		\|\partial^\alpha F\|_{L^\infty(\mathcal B)}\leq \mathfrak{B}^{|\alpha|} M_{|\alpha|}.
		\ees
		We now define a new log-convex sequence $\widetilde {\mathcal M} =\{\widetilde M_k\}_{k\in \N_0}$ as follows
		\be \label{defn-tilde-M}
		\widetilde M_k = \mathfrak{B}^k M_{k}. 
		\ee
		Clearly, $\widetilde{\mathcal M}$ is log-convex, quasi-analytic and $\widetilde M_0=1$.Then it follows that $F\in C_{\widetilde{\mathcal M}}(\mathcal B)$. Moreover, by the assumption on the point \(x_0\), it follows that $|F(x_0)|=  \frac{|f({x_0})|}{C \mathcal{H}^d \|f\|_{L^{2}(B)}} \geq \frac{|\mathcal B|^{-1/2}}{C \mathcal{H}^d}$. We set
		\be \label{defn-s}
		s= min\{1, \; (C \mathcal{H}^d |\mathcal B|^{1/2})^{-1}\}.
		\ee
		Then
		\bes
		\sup_{x\in \mathcal B} |F(x)|\geq s,   \textit{ and hence } \:\: n_{F}\leq n_{\widetilde{\mathcal M},  s}.
		\ees
		Applying Proposition \ref{prop-martin}, we conclude that 
		\bes
		\int_\mathcal B |F(x)|^2~dx \leq \frac{2}{\gamma} \left(\frac{2d}{\gamma} \Gamma(2n_{\widetilde{\mathcal M}, s}) \right)^{4n_{\widetilde{\mathcal M}, s}} \int_{\mathcal{B} \cap \Omega} |F(x)|^2~dx.
		\ees
		Which implies that 
		\be \label{est-local}
		\int_\mathcal B |f(x)|^2~dx \leq \frac{2}{\gamma} \left(\frac{2d}{\gamma} \Gamma(2n_{\widetilde{\mathcal M}, s}) \right)^{4n_{\widetilde{\mathcal M}, s}} \int_{\mathcal{B} \cap \Omega} |f(x)|^2~dx.
		\ee
		Using the estimate (\ref{est-good}) it follows from (\ref{est-local}) and (\ref{est-GB}) that
		\beas
		\int_{X} |f(x)|^2~dx &\leq& 2 \sum_{\{j: \mathcal B_j \textit{ is good }\}} \int_{\mathcal B_j} |f(x)|^2~dx\\
		&\leq& 2 \sum_{\{j: \mathcal B_j \textit{ is good }\}}\frac{2}{\gamma} \left(\frac{2d}{\gamma} \Gamma(2n_{\widetilde{\mathcal M}, s}\right)^{4 n_{\widetilde{\mathcal M}, s}} \|f \|_{L^2(\mathcal B_j \cap \Omega)}^2\\
		&\leq&  \frac{4N}{\gamma} \left(\frac{2d}{\gamma} \Gamma(2n_{\widetilde{\mathcal M}, s}\right)^{4 n_{\widetilde{\mathcal M}, s}} \|f \|_{L^2(\mathcal B_{good} \cap \Omega)}^2\\
		&\leq& \frac{4N}{\gamma} \left(\frac{2d}{\gamma} \Gamma(2n_{\widetilde{\mathcal M}, s}\right)^{4 n_{\widetilde{\mathcal M}, s}} \|f \|_{L^2(\Omega)}^2,
		\eeas
		This completes the proof. 
		
	\end{proof}
	
	\begin{proof}[Proof of Theorem \ref{thm-main-1}]
		As before, $\mathcal B(x, r)$ denotes the geodesic ball centred at $x$ of radius $r$ in $X$. We set $\mathcal R = \text{min}\{r_{inj}, r_0\}$. We choose $r \leq \mathcal R /2$ such that, for each $x\in X$, there exists a quasi-analytic chart  $(\mathcal B(x, r), \phi_x)$ and the quasi-analytic map $\phi_x:\mathcal B(x, r) \ra U_x \subset \R^d$ is a local diffeomorphism. Using compactness of $X$, we find a finite open cover
		\bes
		X = \bigcup_{i=1}^N \mathcal B(x_i,r),
		\ees
		There exists $C>0$ (independent of $r$) such that the $X$ can be covered by at most $N = ([C/{r}]+1)^d$ number of geodesic balls with radius $r$ (see \cite[Lemma 3.1]{KL}). Hence $|\mathcal B(x,r)| N \sim C$ for some $C$ independent of $r$. Since $\omega \subseteq X$ and $|\omega| > 0$, there exist a ball $\mathcal B(x,r)$, with $x = x_i $ for some $i \in \{1,2,...,N\}$, such that
		\be \label{est-lower}
		|\mathcal B(x,r) \cap \omega| \geq \frac{|\omega|}{N}.
		\ee
		Compactness of $X$ also provides $x_0 \in X$ such that $|f(x_0)| = \|f\|_{L^{\infty}(X)}$. Since $X$ is connected, we can join $x_0$ and $x$ by an overlapping chain of balls with radius $\hat r= 2 r$. Since $f$ satisfies doubling condition (\ref{doubling}), we get
		\be \label{est-est} 
		\|f\|_{L^{\infty}(X)} = |f(x_0)| =\|f\|_{L^{\infty}(\mathcal B(x_0, \hat r))} \leq \kappa^{\mathcal C}\|f\|_{L^{\infty}(\mathcal B(x, \hat r))},
		\ee
		where the constant $\mathcal C$ depends on $r$. Using doubling property (\ref{doubling}) once again on concentric balls centered at $x$, we obtain
		\be \label{est-est-2}
		\|f\|_{L^{\infty}(\mathcal B(x, \hat r))} \leq \kappa \|f\|_{L^{\infty}(\mathcal B(x,r))}.
		\ee
		From the above estimates (\ref{est-est}) and (\ref{est-est-2}), we conclude that 
		\be \label{est-est-3} 
		\|f\|_{L^{\infty}(X)} \leq \kappa^{\mathcal C+1} \|f\|_{L^{\infty}(\mathcal B(x,r))}.
		\ee
		Thus, it suffices to prove the theorem for $\mathcal B(x,r)$ instead of $X$. 
		Therefore, we aim to establish an observability estimate for the geodesic ball $\mathcal B(x,r)$. 
		From now onward, we use $\mathcal B$ to denote $\mathcal B(x,r)$. 
		Since $f \in C_{\mathcal M}(X)$, it follows from  Lemma \ref{lem-DR} that there exists $L\geq 1, h>0$ (depending on $\mathcal M$ and $X$)		
		\bes
		\|\partial^\alpha f\|_{L^\infty(X)}\leq  L h^{|\alpha|} M_{|\alpha|} , \; \; \;\text{ for all multi-indices $\alpha \in \mathbb{N}_0^d$},
		\ees
		Consequently, $f\in \widetilde{C}_{\mathcal M}(\mathcal B)$. Therefore, by the Lemma \ref{lem-KL}, we get that
		\be
		\|f\|_{L^\infty(\mathcal B)} \leq \kappa^{2 \mathcal C} {C}^{n_0} \max\{1, h\}^{n_0+1} LM_{n_0+1}  \left(\frac{|\mathcal B|}{|\mathcal B \cap \omega|}\right)^{2n_0}    \|f\|_{L^\infty(\mathcal B \cap \omega)},
		\ee
		Hence, the fact that $|\mathcal B | N \sim C$ and the estimate (\ref{est-lower}) imply that  
		\be
		\|f\|_{L^\infty(\mathcal B)} \leq \kappa^{2 \mathcal C} {C}^{n_0}\max\{1, h\}^{n_0+1} L  M_{n_0+1} \left(\frac{C}{|\omega|}\right)^{2n_0}   \|f\|_{L^\infty(\omega)},
		\ee
		Using the estimates (\ref{est-est-3}), we obtain 
		\bes
		\|f\|_{L^{\infty}(X)} \leq \kappa^{3\mathcal C+1}  \left(\frac{C}{|\omega|}\right)^{2n_0} M_{n_0}  \|f\|_{L^\infty(\omega)}.
		\ees
		where $n_0= 2[\max \{\log_2 \kappa, \mathfrak l e (hr_0)^{-1}\}] +1$ and $C$ is a constant depending only on $\mathcal M$ and $X$, while $\mathcal C$ depends solely on $X$.
	\end{proof}
	
	\begin{proof}[proof of corollary \ref{cor-eig} ]
		We have $h=\sum_{j=1}^N \phi_j$.
		Therefore, Plancherel formula yields
		\beas
		\|P^k h \|_{L^2(X)}^2 &=& \sum_{\lambda_j=1}^N \lambda_j^{2k}\|\phi_j\|_{L^2(X)}^2 \leq  {\lambda_N}^{2k} \|h\|_{L^2(X)}^2	
		\eeas
		Define $M_k = k! (\lambda_N)^{k/m}$, so $\mathfrak m = \{(\lambda_N)^{k/m}\}_{k \geq 0}$.
		Now it easy to check that $\{M_k\}_{k \in \N_0}$ is a regular weight sequence and satisfying \eqref{eqn-mg},\eqref{M-con}. Therefore $h \in C_{\mathcal M}(X)$.
		
		It follows from \cite[Theorem 4.1]{D} and \cite{KL} that there exist $r_0, C>0  $ such that
		\bes
		\sup_{\mathcal B(x,2r)} |h|\leq e^{\gamma} \sup_{\mathcal B(x,r)} |h|, \:\: \textit{ for all } \:\:  r\in (0, r_0], \; \;  x\in X,
		\ees
		where $\gamma = C \sqrt{\lambda_N} + CN^2~\log N + C$.
		Applying Theorem~\ref{thm-main-1}, we obtain 
		\beas
		\|h\|_{L^{\infty}(X)} &\leq& \left(e^{\gamma}\right)^{3\mathcal C +1} \left(\frac{C}{|\omega|}\right)^{2\gamma} M_{n_0} \|f\|_{L^\infty(\omega)} \\
		&\leq& \left( \frac{C}{|\omega|}\right)^{2 \gamma} \gamma^{C \gamma}\|h\|_{L^{\infty}(\omega)}.
		\eeas
		
	\end{proof}

	\begin{proof}[Proof of Theorem \ref{QM}]
		Let $c>0$. We show that there exists $C>0$ depending on $\gamma, c$ and $X$ such that whenever a function $f \in L^2(X)$ satisfies
		\be \label{est-proof}
		\sum_{j\in \N_0} |f_j|^2 ~e^{2j^{1/m}/log^\delta(e+ j^{1/m})} \leq c^2 \|f\|_{L^2(X)}^2,
		\ee
		and $\Omega$ is a $\gamma$-relatively dense set in $X$, we get $\|f\|_{L^2(X)} \leq C \|f\|_{L^2(\Omega)}$.
		In view of Theorem \ref{thm-main}, it is enough to show that $f \in C_{\mathcal M}(X)$, for some regular weight sequence $\mathcal M$ which satisfies moderate growth condition \eqref{eqn-mg} together with \eqref{M-con} and quasi-analytic condition \eqref{quasi-cond}. Let $W(t) = e^{t/log^{\delta}(e+t)}$, for $t > 1$ and $W(t)= 1$, for $t\in [0, 1]$.  Without loss of generality,  we assume $\|f\|_{L^2(X)}=1$. It can be shown that 
		\bes
		\sup_{\lambda \geq 0} \frac{\lambda^k}{W(\lambda)}\leq {\widetilde{C}}^k k!~ (\log(e+k))^{\delta k},
		\ees
		where $\widetilde{C}>0$ is independent of $k$.
		Let us define $M_k = {\widetilde{C}}^k k!~ (\log(e+k))^{\delta k}$. Using Plancherel formula (\ref{Plancherel}) and hypothesis (\ref{est-proof}),  we get
		\beas 
		\|P^{k} f\|_{L^2(X)}&=&  \left(\sum_{j \in \mathbb N} \frac{\lambda_{j}^{2k}}{W(\lambda_j^{1/m})^2} |f_j|^2 ~W(\lambda_j^{1/m})^2 \right)^{1/2}  \\
		&\leq& \sup_{j \in \mathbb N_0} \frac{\lambda_j^{k}}{W(\lambda_j^{1/m})}  \left(\sum_{k \in \mathbb N} |f_j|^2~W(\lambda_j^{1/m})^2 \right)^{1/2} \\
		&\leq& c ~ \sup_{\lambda\geq 0} \frac{\lambda^{mk}}{W(\lambda)}~ \|f\|_{L^2{(X)}} \leq c~M_{mk} \|f\|_{L^2{(X)}}.
		\eeas
		Let $\mathcal M^\prime= \{c M_k\}_{k \in \N_0}$. By \cite[Example~2.14]{Fu}, $\mathcal M'$ is a regular weight sequence satisfying \eqref{eqn-mg} and \eqref{M-con}, with shift modification \(m_0=c>0\). Moreover, it follows from \cite[Proposition~2.2]{JM} that $\mathcal M'$ is quasi-analytic. Hence,
   	    $f\in C_{\mathcal M'}(X)$. This completes the proof of Theorem~\ref{QM}, in view of Theorem~\ref{thm-main}.
   	    
	\end{proof}    
	
	\begin{proof}[Proof of Theorem \ref{unique-cont}]
		Let $f\in C_{\mathcal M}(X)$ and $\omega\subset X$ is a positive measure set on which $f$ vanishes identically. By Lemma~\ref{lem-DR}, it follows that $f \in \widetilde{C}_{\mathcal M}(\mathcal B)$. Let us cover $X\subset \cup_{j=1}^N \mathcal B_j$ where each $\mathcal B_j$ is a ball of radius $r_{inj}/2$. Then there exists a ball $\mathcal B$ such that $|\omega \cap \mathcal B|>0$. We claim that  $f$ must vanishes identically on $\mathcal B$. If not, then $t=\|f\|_{L^\infty(\mathcal B)}>0$ and hence by Proposition \ref{prop-martin} we conclude that $f$ vanishes on the ball $\mathcal B$, which is a contradiction. Repeating the same argument for finite number of overlapping balls, we conclude that $f$ vanishes on $X$.
	\end{proof}

\end{document}